\documentclass[reqno]{amsart}
\usepackage{enumerate}

\usepackage{hyperref}

\newtheorem{theorem}{Theorem}[section]
\newtheorem{lemma}[theorem]{Lemma}

\theoremstyle{definition}
\newtheorem{definition}[theorem]{Definition}

\newtheorem{remark}[theorem]{Remark}

\newtheorem{assumption}[theorem]{Assumption}

\newcommand{\norm}[1]{\left\Vert#1\right\Vert}

\numberwithin{equation}{section}
\usepackage{amsfonts}

\usepackage{amsmath}
\usepackage{amssymb}
\usepackage{cases}
\usepackage{cite}
\usepackage{hyperref}
\usepackage{float}
\hypersetup{
   colorlinks,
    citecolor=blue,
    filecolor=black,
    linkcolor=blue,
    urlcolor=magenta
}

\begin{document}
\font\nho=cmr10
\def\dive{\mathrm{div}}
\def\cal{\mathcal}
\def\L{\cal L}

\def \ud{\underline }
\def\id{{\indent }}
\def\f{\frac}
\def\non{{\noindent}}
 \def\le{\leqslant} 
 \def\leq{\leqslant}
 \def\geq{\geqslant} 
\def\rar{\rightarrow}
\def\Rar{\Rightarrow}
\def\ti{\times}
\def\i{\mathbb I}
\def\j{\mathbb J}
\def\si{\sigma}
\def\Ga{\Gamma}
\def\ga{\gamma}
\def\ld{{\lambda}}
\def\Si{\Psi}
\def\f{\mathbf F}
\def\r{\hro{R}}
\def\e{\cal{E}}
\def\B{\cal B}
\def\A{\mathcal{A}}
\def\p{\mathbb P}

\def\tet{\theta}
\def\Tet{\Theta}
\def\hro{\mathbb}
\def\ho{\mathcal}
\def\P{\ho P}
\def\E{\mathcal{E}}
\def\n{\mathbb{N}}
\def\M{\mathbb{M}}
\def\dMu{\mathbf{U}}
\def\dMcs{\mathbf{C}}
\def\dMcu{\mathbf{C^u}}
\def\vk{\vskip 0.2cm}
\def\td{\Leftrightarrow}
\def\df{\frac}
\def\Wei{\mathrm{We}}
\def\Rey{\mathrm{Re}}
\def\s{\mathbb S}
\def\l{\mathcal{L}}
\def\C+{C_+([t_0,\infty))}
\def\o{\cal O}

\title[Certain classes of mild solutions of Lienard equation]{On certain classes of mild solutions of the scalar Li\'enard equation revisited}

\author[P.T. Xuan]{Pham Truong Xuan}
\address{Pham Truong Xuan \hfill\break Thang Long Institute of Mathematics and Applied Sciences (TIMAS), Thang Long University, Nghiem Xuan Yem, Hanoi, Vietnam} 
\email{phamtruongxuan.k5@gmail.com or xuanpt@thanglong.edu.vn} 

\author[N.T. Van]{Nguyen Thi Van}
\address{Nguyen Thi Van \hfill\break
Faculty of Computer Science and Engineering, Thuyloi University, 175 Tay Son, Dong Da, Hanoi, Vietnam} 
\email{van@tlu.edu.vn}

\author[N.T. Loan]{Nguyen Thi Loan}
\address{Nguyen Thi Loan\hfill\break Faculty of Fundamental Sciences, Phenikaa University, Hanoi 12116, Vietnam.}
\email{loan.nguyenthi2@phenikaa-uni.edu.vn}

\author[T.M. Nguyet]{Tran Minh Nguyet}
\address{Tran Minh Nguyet \hfill\break Department of Mathematics, Thang Long University, Nghiem Xuan Yem, Hanoi, Vietnam}
\email{nguyettm@thanglong.edu.vn}

\begin{abstract}  
In this work we revisit the existence, uniqueness and exponential decay of some classes of mild solutions which are almost periodic (AP-), asymptotically almost periodic (AAP-) and pseudo almost periodic (PAP-) of the scalar Lin\'eard equation by employing the notion of Green function and Massera-type principle. First, by changing variable we convert this equation to a system of first order differential equations. Then, we transform the problem into a framework of an abstract parabolic evolution equation which associates with an evolution family equipped an exponential dichtonomy and the corresponding Green function is exponentially almost periodic. After that, we prove a Massera-type principle that the corresponding linear equation has a uniqueness AP-, AAP- and PAP- mild solution if the right hand side and the coefficient functions are AP-, AAP- and PAP- functions, respectively. The well-posedness of semilinear equation is proved by using fixed point arguments and the exponential decay of mild solutions are obtained by using Gronwall's inequality. Although our work revisits some previous works on well-posedness of the Lin\'eard equation on these type solutions but provide a difference view by using Green function and go farther on the aspects of asymptotic behaviour of solutions and the construction of abstract theory. Our abstract results can be also applied to other parabolic evolution equations.
\end{abstract}



\maketitle

\tableofcontents

\section{Introduction}
Over the last century, various models of the Li\'enard type systems have been widely studied due to important applications in engineering, physics,mechanics and biology \cite{Gao,Xu,Xu1,Pe,Liu,GaGu}.  Historically, in the late 1920s, Liénard investigated the equation
\begin{equation}\label{eq01}
x''+\epsilon f(x)x'+\omega^2 x=0,
\end{equation}
in related to electrical circuits \cite{Li}, where $f(x)$ is even, $\epsilon$ is a positive scaling parameter, and $\omega$ is a constant. In the case, when $ f(x)=x^2-1$ and $ \omega =1$,  the equation \eqref{eq01} becomes the van der Pol oscillator, which  is a  popular model for simulating the dynamics of nonlinear physical phenomena. The Liénard systems have been motivated from mathematicians in the generation form   
\begin{equation}\label{eq02}
x''+\epsilon f(x)x'+ g(x)=e(t).
\end{equation} 

It is very significant to have information about the qualitative behavior of solutions of these type systems. Many studies on qualitative behaviors of solutions such as periodicity, oscillation, almost periodicity and pseudo almost periodicity are quite common in \cite{Gao,Liu,Pe,Xu,Xu1,Ya}. One should also recall that the concept of almost periodicity  was introduced in the first time by Bohr in the mid- twienties (see \cite{Bo1, Bo2, Bo3}).Then, the theory of almost periodic functions was continuously getting developement by some mathematicians like Amerio and Prouse \cite{Ame}, Levitan \cite{Le}, Besicovitch \cite{Be}, Bochner\cite{Boc}, von Neumann, Fr\'echet, Pontryagin, Lusternik, Stepanov, Weyl, etc. (see \cite{Che, Dia}). The concept of asymptotically almost periodicity was introduced later by the French mathematician Fr\'echet. There is series of works of known authors dedicated to asymptotically almost periodicity of solutions of differential equations, see \cite{Dia} for details. Besides, Zhang\cite{Zhang1} gave the concept pseudo almost periodic function to prove the existence of these solutions for a nonlinear parabolic equation.  When we review the literature, we can discover some works on qualitative behaviors of solutions. Fink \cite{Fi} used the $L^1-$ norm as a  Lyapunov function to prove the existence of almost periodic solutions for the above equation in condition $ 0<2\inf g'(x)\leqslant 2\sup g'(x)\leqslant\inf f^2(x) \leqslant  \sup f^2(x)<\infty$. Langenhop and Seifert \cite{LaSe} have investigated the existence of almost periodic solutions for the  forced Li\'enard equation contained in the bounded region $\Omega\subset \mathbb{R}^2$ and the solution is asymptotically stable with respect to all solutions whose trajectories enter $\Omega$ for all time. Moreover, Caraballo and D. Cheban (\cite{Ca}) studied the existence and uniqueness of almost periodic and asymptotically almost periodic mild solutions for the Li\'enard equation \eqref{eq02}. In a series of works \cite{Cie1,Cie2,Cie3}, Cieutat et al. established the structure of the set of bounded solutions and existence of certain types of solutions (such as almost periodic, pseudo almost periodic, etc...) both for the scalar and vectorial Li\'enard equations.
In addition, Gao and Liu \cite{Gao} studied the following almost periodic solutions of the following multiple-delay Li\'enard equation
\begin{equation}\label{eq03}
x''+ f(x(t))x(t)'+ g(x(t))+ \sum_{i=1}^n h_i(x(t-\sigma_i(t))=e(t),
\end{equation} 
where $f$ and $h_i$ (where $i = 0,1,2,...,n$) are continuous functions on $\mathbb{R}$, and the functions $\sigma_i(t) \leqslant 0$ (where $i = 1, 2, . . . , n$) and $e(t)$ are almost periodic on $\mathbb{R}$. Yazgan continued to show the existence of solutions for the Li\'enard -type system with multiple-delay   in the case the solutions are pseudo almost periodic (see \cite{Ya1}) and  more general, weighted pseudo almost automorphic by utilizing some differential inequalities, the main features of the weighted pseudo almost periodic and Banach fixed point theorem. 

Since the previous works (see for examples \cite{Gao,Xu,Zhu}), by changing variables method the scalar Li\'enard-type equations can be transformed into an abstract framework of nonautonom parabolic evolution equations on a Banach space $X$ and on the whole-line time axis as (see Subsection \ref{22}):
\begin{equation}\label{eq1}
u'(t)=A(t)u(t) + h(t,u(t)) \hbox{      for   } t\in \mathbb{R}
\end{equation}
and on the half-line time axis as
\begin{equation}\label{eq2} 
\left\{
  \begin{array}{rll}
u'(t) \!\! &= A(t)u(t) + h(t,u(t)) \quad  & \hbox{    for   } t\in \mathbb{R}_+, \hfill \cr
u(0) \!\!&=\; u(0) \in X. \quad &{}\cr
\end{array}\right.
\end{equation}
There is a rich literature which studies qualitative theory of parabolic evolution equations such as well-posedness of mild solutions (and certain types of mild solutions) and their stability under certain suitable conditions of operator $A(t)$ (see \cite{Batty1999,Schnau,MaSch,Naito,Phong,Ruess1} and many reference thereins). Therefore, in the present paper we shall use and develop the abstract results to revisit the original scalar Li\'enard equation (and its-type equations). This approach provides a general view in the study. We want to express in this paper the thought process in the mind (which reflects the philosophy process of perception): we start from the study of existence and stability of mild solutions for the scalar Li\'enard equation. These problems lead us to come to study on an abstract parabolic evolution equation. Then, we comeback to apply abstract results to the origin equation. In addition, we find also that the abstract results can be applied to some nonautonom parabolic equations and other second-order equations. This process is clearly demonstrated through the structure of this article.

The existence, uniqueness and stability for almost periodic mild solutions and its generalizations for parabolic evolution equations (and the other equations which can be converted to parabolic evolution equations) have been extensively studied for a long time. We refer the readers to some useful books \cite{Ame,Dia,Fink,Hale,Le} which provide the theory on almost periodic functions and its applications in PDE. 
In our knowledge, most of the works in this topic for nonautonom parabolic evolution equations the authors need the exponential dichotomy of evolution family $\left\{ U(t,s)\right\}_{t \geqslant s}$ and some suitable conditions on the operator $A(t)$ (which is associated with the equations) to prove the existence of certain types of mild solutions. In particular, Schnaubelt et al. \cite{Schnau} provided an important result about the equivalence between the existence an evolution family (which satisfies exponential bounded condition and exponential dichtonomy) and the existence of bounded mild solutions for nonautonom parabolic evolution equations.
Then, Maniar and Schnaubelt \cite{MaSch} studied the existence of almost periodic (defined on the whole-line time axis) and asymptotically almost periodic (defined on the half-line time axis) mild solutions for non homogeneous parabolic evolution equations. To do this, the authors in \cite{MaSch} considered that there was an exponential bounded evolution family $\left\{ U(t,s)\right\}_{t\geqslant s}$ which satisfies exponential dichtonomy and they assume further that the resolvent of operator $A(t)$ (associated with the corresponding linear equations) is an almost periodic function with respect to the time. In the same time, Batty et al. \cite{Batty1999} showed the existence results about asymptotically almost periodic mild solutions for nonautonom parabolic equations with asymptotically almost periodic external forces under the assumptions on the periodic condition of operator $A(\cdot)$ and the spectrum of the monodromy operator $V=U(q,0)$ containing only countably many points of the unit circle. In fact, the results in \cite{Batty1999} generalized from the early ones obtained in \cite{Aren,Phong,Ruess1}. In this direction, we refer other related to  Naito and Minh \cite{Naito}. The periodic assumption of operator $A(\cdot)$ or $U(\cdot,\cdot)$ was also used in other works concerning the existence of periodic mild solutions for parabolic evolution equations and more specific for the partial neutral functional differential equations (see for examples \cite{Huy2016,Huy2019}). 

On the other hand, there are many works which introduce the notions of certain generalized types of mild solutions (such as pseudo almost periodic, weighted pseudo almost periodic, almost automorphic, pseudo almost automorphic, etc... mild solutions) and then study the existence and uniqueness  of such solutions for autonom and nonautonom parabolic equations (see \cite{Ami,Baro,Baro1,Cu,Dia1,
Dia2,Dia3,Zhang1,Zhang2} and many reference thereins). We refer the readers to useful books \cite{Dia,Gu, LiuMinh} which provide fully the notions of generalized functions and their applications to parabolic evolution equations and dynamical systems.

Now, we describe our method and results as follows: for a given function $v\in C_b(\mathbb{R},X)$, the corresponding nonhomogenuous linear equations \eqref{eq1} and \eqref{eq2} are
\begin{equation}\label{eq11}
u'(t)=A(t)u(t) + h(t,v(t)) \hbox{      for   } t\in \mathbb{R}.
\end{equation}
and on the half-line time axis
\begin{equation}\label{eq22} 
\left\{
  \begin{array}{rll}
u'(t) \!\! &= A(t)u(t) + h(t,v(t)) \quad  & \hbox{    for   } t\in \mathbb{R}_+, \hfill \cr
u(0) \!\!&=\; u(0) \in X, \quad &{}\cr
\end{array}\right.
\end{equation}
respectively. 

In order to establish the existence and exponential stability, we make some necessary assumptions that: there is an evolution family $\left\{ U(t,s) \right\}_{t\geqslant s}$ associated with the homogeneous equation $u'(t)=A(t)u(t)$ and $U(t,s)$ satisfying exponential dichtonomy and exponential boundedness in Assumption \ref{dich} below. Under the exponential almost periodic assumption of the Green function $G(t,s)$ (see Assumption \ref{GreenAss} below), we prove the well-posedness of almost periodic (AP-), asymptotically almost periodic (AAP-) and pseudo almost periodic (PAP-) mild solutions for the corresponding linear equations of \eqref{eq11} and \eqref{eq22} by proving Massera-type principle that: if $h$ is AP- or AAP- or PAP-function, then the mild solution is also AP- or AAP- or PAP-function, respectively (in details see Theorem \ref{lem2}). Note that, we will use the formula of mild solutions via the Green function $G(t,s)$ in the proofs, meanwhile most of the previous works used the direct formula of mild solutions via $U(t,s)$. In comparing with the previous works on the Li\'enard equations, our novelty is the exponential almost periodic assumption of $G(t,s)$ (see Assumption \ref{GreenAss}) which covers all of conditions used in previous works to guarantee the validity of Massera-type principle for the scalar Li\'enard-type equations (see for examples \cite{Gao,Xu} and reference thereins). Since we do not consider directly the periodic or almost periodic conditions on the operator $A(t)$ or $U(t,s)$, this is also a difference from the previous works on the existence of certain types of mild solutions (such as AP-,AAP-, PAP- and other types of mild solutions (for a monograph for introduction, see e.g., \cite{Dia}).

Using the well-posedness results of linear equations and fixed point arguments we establish the existence of AP-, AAP- and PAP- mild solutions for semilinear equations \eqref{eq1} and \eqref{eq2} (see Theorem \ref{th:2}). The exponential stability will be proven by utilizing the exponential stability of the Green function and a Gronwall-type inequality. Finally, we construct some examples for applying the abstract results to the scalar Li\'enard equation (see Subsection \ref{APP}) and we provide also another example for non-autonom parabolic equation (see Subsection \ref{S52}).

Our paper is organized as follows: Section \ref{S2} relies on the certain concepts of generalized functions such as almost periodic, asymptotically almost periodic and pseudo almost periodic solutions (see Subsection \ref{S21}) and on the formulas of the scalar Li\'enard equation (see Subsection \ref{22}). In Section \ref{S3} we prove the well-posedness results for the linear parabolic evolution equations. In Section \ref{S4}, we study the existence and exponential stability of AP-, AAP- and PAP- mild solutions for semilinear equations. Finally, we apply the abstract results obtained in previous sections to the scalar Li\'enard equation and a nonautonom parabolic equation in  Section \ref{S5}.

\section{Preliminaries}\label{S2}
\subsection{The concepts of functions}\label{S21}
\bigskip
Let $X$ be a Banach space, we denote 
$$C_b(\mathbb{R}, X):=\{f:\mathbb{R} \to X \mid f\hbox{ is continuous on $\mathbb{R}$  and }\sup_{t\in\mathbb{R}}\|f(t)\|_X<\infty\}$$
which is a Banach space endowed with the norm $\|f\|_{\infty, X}=\|f\|_{C_b(\mathbb{R}, X)}:=\sup\limits_{t\in\mathbb{R}}\|f(t)\|_X$.
Similarly, we denote 
$$C_b(\mathbb{R}_+, X):=\{f:\mathbb{R}_+ \to X \mid f\hbox{ is continuous on $\mathbb{R}$ and }\sup_{t\in\mathbb{R}_+}\|f(t)\|_X<\infty\}$$
which is also a  Banach space endowed with the norm $\|f\|_{\infty, X}=\|f\|_{C_b(\mathbb{R}_+, X)}:=\sup\limits_{t\in\mathbb{R}_+}\|f(t)\|_X$.

\begin{definition}(Bohr \cite{Bo1,Bo2,Bo3})
A function  $h \in C_b(\mathbb{R}, X )$ is called almost periodic function if for each $ \epsilon  > 0$, there exists $l_{\epsilon}>0 $ such that every interval of length $l_{\epsilon}$ contains at least a number $T $ with the following property
\begin{equation*}
 \sup_{t \in \mathbb{R} } \| h(t+T)  - h(t) \| < \epsilon.
\end{equation*}
The collection of all almost periodic functions $h:\mathbb{R} \to X $ will be denoted by $AP(\mathbb{R},X)$ which is a Banach space endowed with the norm $\|h\|_{ AP(\mathbb{R},X)}=\sup\limits_{t\in\mathbb{R}}\|h(t)\|_X.$
\end{definition}
To introduce the asymptotically almost periodic functions, we need the space  $C_0 (\mathbb{R}_+,X)$, that is, the collection of all continuous functions $\varphi: \mathbb{R}_+ \to X$ such that
 $$\mathop{\lim}\limits_{t \to \infty } \| \varphi(t) \|=0.$$
Clearly, $C_0 (\mathbb{R}_+,X)$  is a Banach space endowed with the norm $\|\varphi\|_{C_0 (\mathbb{R}_+,X)}=\sup\limits_{t\in\mathbb{R}_+}\|\varphi(t)\|_X$.
\begin{definition} 
A function  $f \in C_b(\mathbb{R}_+, X )$  is said to be (forward-) asymptotically almost periodic if there exist  $h \in AP(\mathbb{R},X)$ and $ \varphi\in C_0(\mathbb{R}_+,X)$ such that
\begin{equation*}
f(t) = h(t) + \varphi(t).
\end{equation*}
We denote 
$$AAP(\mathbb{R}_+, X):= \{f:\mathbb{R}_+ \to X \mid f\hbox{ is asymptotically almost periodic on $\mathbb{R}_+$}\}.$$ Under the norm $\|f\|_{ AAP(\mathbb{R}_+,X)}=\|h\|_{ AP(\mathbb{R}, X)}+\|\varphi\|_{ C_0(\mathbb{R}_+,X)}$, then $AAP(\mathbb{R}_+,X)$ is a Banach space.
\end{definition} 
The decomposition of asymptotically almost periodic functions is unique (see \cite[Proposition 3.44, page 97]{Dia}), that is, 
$$AAP(\mathbb{R}_+,X) = AP(\mathbb{R},X) \oplus C_0(\mathbb{R}_+,X).$$
\begin{definition}\label{WPAAfunction}(PAP-function)
A function $f \in C_b(\mathbb{R},X)$ is called pseudo almost periodic if it can be decomposed as $f = g + \phi$ where $g \in AP(\mathbb{R},X)$ and $\phi$ is a bounded continuous function with vanishing mean value i.e
\begin{equation*}\label{meanvalue}
\lim_{L\to \infty}\frac{1}{2L}\int_{-L}^L \|\phi(t)\|_X dt =0.
\end{equation*}
We denote the set of all functions with vanishing mean value by $PAP_0(\mathbb{R},X)$
and the set of all the pseudo almost periodic (PAP-) functions by $PAP(\mathbb{R},X)$. 
\end{definition}
We have that $(PAP(\mathbb{R},X),\|.\|_{\infty,X})$ is a Banach space, where $\|.\|_{\infty,X}$ is the supremum norm (see \cite[Theorem 5.9]{Dia}). As well as AAP- functional space, we have the following decomposition (see also \cite{Dia}): 
$$PAP(\mathbb{R},X) = AP(\mathbb{R},X) \oplus PAP_0(\mathbb{R},X).$$

The notion of pseudo almost periodic function is a generalisation of the almost periodic and asymptotically almost periodic (AAP-) and almost periodic (AP-) functions. Precisely, we have the following inclusions
$$P(\mathbb{R},X) \hookrightarrow AP(\mathbb{R},X) \hookrightarrow AAP(\mathbb{R}_+,X) \hookrightarrow PAP(\mathbb{R},X) \hookrightarrow C_b(\mathbb{R},X).$$
where $P(\mathbb{R},X)$ is the space of all  continuous and periodic functions from $\mathbb{R}$  to $X$.\\

{\bf Example.}
\begin{itemize}
\item[$(i)$] The function $h(t)=\sin{ t}+\sin({\sqrt{2}t})$ is almost periodic but not periodic,  $\tilde{h}(t) =\sin{ t}+\sin({\sqrt{2}t})+ \dfrac{1}{|t|}$  is asymptotically almost periodic but not almost periodic.
\item[$(ii)$] {The function $\hat{h}(t) = \sin{ t}+\sin({\sqrt{2}t}) + \dfrac{t}{\sqrt{1+t^2}}$ is pseudo almost periodic but not asymptotically almost periodic.}
\item[$(iii)$] Let $X$ be a Banach space and $g\in X - \left\{ 0\right\}$, we have that $f=hg\in AP(\mathbb{R},X)$, $\tilde{f}=\tilde{h}g \in AAP(\mathbb{R}_+,X)$ and $\hat{f} = \hat{h}g\in PAP(\mathbb{R},X)$.
\end{itemize}

\subsection{From the scalar Li\'enard equation to an abstract framework} \label{22}
We consider the following original scalar Li\'enard-type equation
\begin{equation}\label{eq:0.1}
x'' + {f(x)}x'+g(t,x)=e(t),
\end{equation}
where {$f:\mathbb{R}\to \mathbb{R}$} (depending on $x$), $g:\mathbb{R}^2\to \mathbb{R}$ (depending on $(t,x)$) and $e:\mathbb{R}\to \mathbb{R}$ (depending on the time $t$) are given functions and $x:\mathbb{R}\to \mathbb{R}$ (depending on the time $t$) is unknown.

The scalar Li\'enard equation \eqref{eq:0.1} is a differential equation with second order. There are several ways to transform this equation to a first order differential system. Below, we follows the changing variables used in \cite{LaSe,Zhu}. Another way to do this can be found in the end of Subsection \ref{APP} below or more details in \cite{Xu}.
Setting $F(x)=\int\limits_0^x f(s)ds$, then the system \eqref{eq:0.1} is equivalent to the following system
\begin{equation*}\label{eq:2.2}
\begin{cases}
    x'=y-\int\limits^x_0 f(s)ds = y-F(x),\\
     y'=-g(t,x)+e(t).
 \end{cases}
\end{equation*}
Under the matrix term, this system can be rewritten as
\begin{equation}\label{eq:2.2matrix}
\frac{d}{dt}\begin{bmatrix}
x\\ 
y
\end{bmatrix}
=A(t)\begin{bmatrix}
 x\\
y
\end{bmatrix} + \begin{bmatrix}
0\\ {-g(t,0)} + e(t)
\end{bmatrix},
\end{equation}
where
\begin{equation*}
A(t) = \int\limits_0^1
\begin{bmatrix}
-f(s{\alpha(t)}) &1\\
-g_x(t,s{\alpha(t)}) &0
\end{bmatrix}
ds.
\end{equation*}
The operator $A(t)$ acts on a vector $\begin{bmatrix}
x\\ 
y
\end{bmatrix}$ by the following sense  
\begin{equation*}
f(s\alpha(t))x \mapsto f(sx(t)), \, g(t,s\alpha(t))x \mapsto g(t,sx(t)).
\end{equation*}

For convenience to state and prove the main results of this paper, we consider the well-posedness and asymptotic behaviour of the following abstract parabolic evolution equation
\begin{equation}\label{abstract}
u'(t) = A(t)u(t) + h(t,u).
\end{equation} 
It is clearly that equation \eqref{eq:2.2matrix} is a specific case of \eqref{abstract}, in fact a linear equation  with $u = \begin{bmatrix}x\\y \end{bmatrix}$ and $h(t,0) = \begin{bmatrix}0\\{-g(t,0)}+e(t)\end{bmatrix}$.

For a given scalar vector function $v$, the corresponding (inhomogenoues) linear equation of \eqref{eq:2.2matrix} is
\begin{equation}\label{linearAbstract}
u'(t) = A(t)u(t) + h(t,v).
\end{equation}

\begin{definition}\label{EvolutionFa}
Let $X$ be a Banach space. A set $U=\left\{ U(t,s): t\geqslant s; t,s\in \mathbb{R}\right\}$ of bounded linear operators on $X$ is called a exponential boundedness evolution family if
\begin{itemize}
\item[$(i)$] $U(t,s)=U(t,r)U(r,s)$ and $U(s,s)=\mathrm{Id}$ for $t\geqslant r\geqslant s$,
\item[$(ii)$] $(t,s)\mapsto U(t,s)$ is strongly continuous for $t\geqslant s$,
\item[$(iii)$] $\|U(t,s)\|\leqslant Me^{\omega (t-s)}$ for $t\geqslant s$.
\end{itemize}
\end{definition}
 
\begin{definition}
An evolution family $\left\{ U(t,s)\right\}_{t\geqslant s}$ is said to solve the Cauchy problem of equation \eqref{linearAbstract} with the initial data $u(s)=u_s \in D(A(s))$ if $u(\cdot)=U(\cdot,s)u_s$ is differentiable, $u(t)\in D(A(t))$ for $t\geqslant 0$ and \eqref{linearAbstract} holds.
\end{definition}
\begin{remark}
We notice that Acquistapace and Terreni \cite{AcTe} introduced the assumptions on the operator $A(t)$ and its resolvent which guarantee that there exists a unique evolution family $U(t,s)$ (satisfying conditions $(i)$ and $(ii)$ in Definition \ref{EvolutionFa}) on $X$ and $U(t,s)$ satisfies also exponential boundedness (condition $(iii)$ in Definition \ref{EvolutionFa}) with $\omega=0$. Moreover, this family solves homogenoues equation $u'(t)=A(t)u(t)$.
\end{remark}
In order to study the existence and uniqueness of mild solutions for semilinar equation \eqref{abstract}, we need the following assumptions on evolution family $\left\{U(t,s) \right\}_{t\geqslant s}$ which is called exponential dichtonomy (see for example \cite{Schnau}).
\begin{assumption}\label{dich}
The evolution family $\left\{ U(t,s) \right\}_{t\geqslant s}$ on Banach space $X$ satisfies exponential dichtonomy on the whole-line time axis $\mathbb{R}$ (with a constant $\beta >0$) if there exists a projection valued function $P:\mathbb{R}\to B(X)$ (where $B(X)$ is the set of bounded linear operators on $X$) such that the function $t\mapsto P(t)z$ is uniformly bounded and strong continuous in $t$ for each $z\in X$ and for some constant $M=M(\beta)>0$ and for all $t\geqslant s$ the following properties hold
\begin{itemize}
\item[$(i)$] $P(t)U(t,s)=U(t,s)P(s)$,
\item[$(ii)$] The restriction $U_Q(t,s) = Q(t)U(t,s)Q(s): Q(s)X \to Q(t)X$ of $U(t,s)$ (where $Q(t)= \mathrm{Id} - P(t)$) is invertible and we set $U_Q(s,t)= (U_Q(t,s))^{-1}$;
\item[$(iii)$] $\left\|U_P(t,s)\right\| \leqslant N e^{-\beta(t-s)}$, where $U_P(t,s) = P(t)U(t,s)P(s)$;
\item[$(iv)$] $\left\|U_Q(s,t)\right\| \leqslant N e^{-\beta(t-s)}$.
\end{itemize}
\end{assumption}
\begin{remark}
\begin{itemize}
\item[$(i)$] The projections $P(t), \, t\in \mathbb{R}$, are called the dichtonomy projections (see \cite{MaSch} for the detailed formula of these projections), and the constants $M,\beta$ are called the dichtonomy constants. 
\item[$(ii)$] Clearly, we have $H=\sup\limits_{t\in \mathbb{R}}\left\|P(t)\right\|<\infty$.
\end{itemize}
\end{remark}
For $U(t,s)$ satisfied Assumption \ref{dich}, one can define the Green function as follows
\begin{equation}\label{GreenForm}
G(t,s)=\begin{cases}
    P(t)U(t,s) & \hbox{for  } t>s,\, t,s\in \mathbb{R} \cr
    -U(t,s)Q(s) & \hbox{for  } t\leqslant s,\, t,s \in \mathbb{R}.
 \end{cases}
\end{equation} 
Clearly, using $(iii)$ and $(iv)$ in Assumption \ref{dich} we have
\begin{equation}\label{ineq-GF}
\|G(t,s)\| \leqslant (1+H)N e^{-\beta |t-s|} 
\end{equation}
for $t\neq s$.

To establish the well-posedness of the almost periodic (AP-), asymptotically almost periodic (AAP-) and pseudo almost periodic (PAP-) mild solutions for linear and semilinear equations \eqref{abstract} and \eqref{linearAbstract}, we need the following assumption on the Green function.
\begin{assumption}\label{GreenAss}
{The Green function $G(t,s)$ is said to exponentially almost periodic if it satisfies: 
for given constants $\eta>0$ and $\varepsilon>0$, there exist constants $\gamma>0$ and $l_{\varepsilon}>0 $ such that every interval of length $l_{\varepsilon}$ contains at least a number $T $ such that}
\begin{equation}\label{GreenAP}
\|G(T+t,T+s)-G(t,s)\| \leqslant \varepsilon e^{-\gamma(t-s)}
\end{equation}
for $|t-s|\geqslant \eta>0$.
\end{assumption}
\begin{remark}\label{rem}
\begin{itemize}
\item[$(i)$] Clearly, this assumption is valid if $A(t)$ is periodic, i.e., there exists a constant $T>0$ such that $A(t+T)=A(t)$. This leads to the periodicity of $R(\delta,A(\cdot))$, then $G(t,s)$ is periodic with the same periodicity $T$, i.e., $G(T+t,T+s)=G(t,s)$ (see \cite{MaSch}).
Therefore, $G(t,s)$ satisfies inequality \eqref{GreenAP} for $\varepsilon=0$.

\item[$(ii)$] Moreover, this assumption is valid if we assume that the operator $A(\cdot)$ (or more general the resolvent $R(\delta, A(\cdot)))$ is almost periodic (see the proof in \cite{MaSch,Yagi}).
\end{itemize}
\end{remark}

\section{Well-posedness of AP-, AAP- and PAP- mild solutions for linear equations}\label{S3}
The mild solution of linear evolution equation \eqref{linearAbstract} on the whole-line time axis $\mathbb{R}_t$ is defined by
\begin{equation*}\label{mild1}
u(t) = U(t,s)u(s) + \int_s^t U(t,s)h(s,v(s))ds \hbox{   for   } t\geqslant s.
\end{equation*}
On the other hand, if we consider Cauchy problem of equation \eqref{linearAbstract} with the initial data $u(0)=u_0$, then the mild solution (on the half-line time axis) is defined by
\begin{equation}\label{mild2}
u(t) = U(t,0)u_0 + \int_0^t U(t,s)h(s,v(s))ds.
\end{equation}
{There are many previous works which used these direct formulas to establish the existence of AP-, AAP- an PAP- mild solutions (and other types of mild-solutions) for linear equation \eqref{linearAbstract} (see for examples \cite{Gao,Zhu,Xu,Ya} and reference thereins). However, below we will provide another approach by employing the formula of mild solutions representing via Green function. This approach helps us to give the proofs more shorter than the ones in previous works.}

We state and prove the main result of this section in the following theorem.
\begin{theorem}\label{lem2} 
The following assertions hold
\begin{itemize}
\item[$(i)$] If $h\in C_b(\mathbb{R},X)$ (resp. $h\in C_b(\mathbb{R}_+,X)$), then linear equation \eqref{linearAbstract} has a unique bounded mild solution $\hat{u}\in C_b(\mathbb{R},X)$ (resp. $\hat{u}\in C_b(\mathbb{R}_+,X)$). 
\item[$(ii)$] Assume that the Green function satisfies Assumptiton \ref{GreenAss}. Then, the Massera-type principle is valid for AP-, AAP- and PAP- mild solutions. Precisely, the following assertions hold
\begin{itemize}
\item[$\bullet$] If $v(\cdot)\in AP(\mathbb{R},X)$ and $h(\cdot,v(\cdot))\in AP(\mathbb{R},X)$, then equation \eqref{linearAbstract} has a unique mild solution $\hat{u}\in AP(\mathbb{R},X)$.
\item[$\bullet$] If $v(\cdot)\in AAP(\mathbb{R}_+,X)$ and $h(\cdot,v(\cdot))\in AAP(\mathbb{R}_+,X)$, then equation \eqref{linearAbstract} has a unique  mild solution $\hat{u}\in AAP(\mathbb{R}_+,X)$.
\item[$\bullet$] If $v(\cdot)\in PAP(\mathbb{R},X)$ and $h(\cdot,v(\cdot))\in PAP(\mathbb{R},X)$, then equation \eqref{linearAbstract} has a unique mild solution $\hat{u}\in PAP(\mathbb{R},X)$.
\end{itemize}
\end{itemize}
\end{theorem}
\begin{proof}
$(i)$ 
{
The well-posedness of bounded mild solutions of linear equation \eqref{linearAbstract} was proven by Schnaubelt et al. in \cite[Theorem 1.1]{Schnau}.
Note that, this result provides also the reason which the necessity of the exponential bounded evolution family $U(t,s)$ in Assumption \ref{dich}.
The formula of bounded mild solution via Green function on the whole line is (see \cite[Pages 496-498]{Schnau}):
\begin{equation*}
u(t) = \int_\mathbb{R} G(t,s)h(s,v(s))ds.
\end{equation*}}
{Due to the exponential stability \eqref{ineq-GF} of Green function, we can estimate easly that
\begin{eqnarray}\label{bounded}
 \|u\|_{C_b(\mathbb{R},X)} &\leqslant& \int_\mathbb{R}\|G(t,s)h(s,v(s))\|_Xds\cr
 &\leqslant& \int_{-\infty}^{+\infty} (1+H)Ne^{-\beta|t-s|}\|h(s,v(s))\|_Xds\cr
 &\leqslant& \frac{2(1+H)N}{\beta} \|h(\cdot, v(\cdot)) \|_{C_b(\mathbb{R},X)}.
\end{eqnarray}}

{Now we consider the Cauchy problem of linear equation \eqref{linearAbstract} with the initial data $u(0)=u_0$. 
Assume that $h\in C_b(\mathbb{R}_+,X)$. Let $u\in C_b(\mathbb{R}_+,X)$ be a solution of this Cauchy problem, then we can rewritten $u$ as (see \cite[Lemma 4.2 (a)]{Huy2016}):
\begin{equation}\label{mildHalf}
u(t) = U(t,0)\zeta_0 + \int_0^\infty G(t,s)h(s.v(s))ds 
\end{equation}
for $\zeta_0 = u_0 - \int_0^\infty G(0,s)h(s,v(s))ds  \in P(0)X$. The converses is also true, i.e., if a function $u$ satisfies integral equation \eqref{mildHalf}, then it is a bounded mild solution of Cauchy problem of linear equation \eqref{linearAbstract} and satisfies integral equation \eqref{mild2}. Therefore, the existence is clearly. Now, we consider $u$ and $v$ are two bounded mild solutions of \eqref{linearAbstract} with $u(0)=v(0)$, then $z=u_1-u_2$ is a bounded mild solution of the following equation
\begin{equation*}
z'(t) = A(t)z(t), \, z(0)=0.
\end{equation*}
Therefore,
\begin{equation*}
z(t) = U(t,0)z(0)=0 \hbox{  for all  } t>0.
\end{equation*}
The uniqueness holds.
By the same way as \eqref{bounded} we can estimate 
\begin{eqnarray}
\|u(t)\|_X &\leqslant& N\|\zeta_0\|_X + \frac{2(1+H)N}{\beta} \|h(\cdot, v(\cdot)) \|_{C_b(\mathbb{R}_+,X)} \nonumber\\
&\leqslant& N\|u_0 - \int_0^\infty G(0,s)h(s,v(s))ds\|_X \nonumber\\
&&\quad+ \frac{2(1+H)N}{\beta} \|h(\cdot, v(\cdot)) \|_{C_b(\mathbb{R}_+,X)}\nonumber\\
&\leqslant& N\|u_0\|_X + (1+H)N\int_0^\infty e^{-\beta s}\|h(s,v(s))\|_X ds \nonumber\\
&&\quad+ \frac{2(1+H)N}{\beta} \|h(\cdot, v(\cdot)) \|_{C_b(\mathbb{R}_+,X)} \nonumber\\
&\leqslant& N\|u_0\|_X + \frac{3(1+H)N}{\beta} \|h(\cdot, v(\cdot)) \|_{C_b(\mathbb{R}_+,X)}.
\end{eqnarray} 
}

$(ii)$ In fact, this assertion verifies the so-called Massera-type principle for parabolic evolution equations (see for example \cite{Dia}). The Massera-type principle was also used to study the well-posedness of various mild solutions in fluid dynamics (see \cite{XVQ,XVT} and many reference thereins). In particular, the principle states that: if the function $f$ satisfies AP or AAP or PAP property, then this is valid for the corresponding mild solution. To prove this assertion we will extend the previous proofs in \cite{XVQ,XVT} to our framework of abstract evolution equation \eqref{linearAbstract}. 

{\bf Issue 1. First, we prove the assertion for almost periodic case.}
From the existence and uniqueness of mild solution on the whole-line time axis obtained in Assertion $(i)$, the solution operator $S:C_b(\mathbb{R},X)\to C_b(\mathbb{R},X)$ associating with the linear equation \eqref{linearAbstract} can be defined as follows: for a given $v\in C_b(\mathbb{R},X)$, we define 
\begin{equation*}\label{SolOpe1}
S(v)(t) =u(t) = \int_{\mathbb{R}}G(t,s)h(s,v(s))ds.
\end{equation*}
Since $v$ is almost periodic and $h(\cdot,v(\cdot))\in AP(\mathbb{R},X)$, we have that: for each $ \varepsilon  > 0$, there exists $l_{\varepsilon}>0 $ such that every interval of length $l_{\varepsilon}$ contains at least a number $T $ with the following property
$$\sup_{t \in \mathbb{R} } \| h(t+T,v(t+T))  - h(t,v(t)) \|_X < \varepsilon.$$
Therefore, combining with the fact that $G(t,s)$ satisfied Assumption \ref{GreenAss}, we can estimate
\begin{eqnarray}\label{AP1}
&&\left\|S(v)(t+T) - S(v)(t)\right\|_X \cr
&=& \left\|\int_{\mathbb{R}} G(T+t,s)h(s,v(s))ds - \int_{\mathbb{R}} G(t,s)h(s,v(s)) ds \right\|_X \cr
&=& \left\|\int_{\mathbb{R}} G(T+t,T+s)h(T+s,v(T+s))ds - \int_{\mathbb{R}} G(t,s)h(s,v(s)) ds \right\|_X \cr
&&\text{ (we used the changing variable   $s\mapsto T+s$ in the first term)}\cr
&\leqslant& \int_{\mathbb{R}} \left\|G(T+t,T+s)[h(T+s,v(T+s)) - h(s,v(s))] \right\|_X ds  \cr
&&+ \int_{|t-s|\geqslant \eta} \left\|[G(T+t,T+s)-G(t,s)]h(s,v(s)) \right\|_X  ds \cr
&&+ \int_{|t-s| < \eta} \left\|[G(T+t,T+s)-G(t,s)]h(s,v(s)) \right\|_X  ds \cr 
&\leqslant& \left\| h(\cdot+T,v(\cdot+T)) - h(\cdot,v(\cdot))\right\|_{C_b(\mathbb{R},X)} \int_{\mathbb{R}} (1+H)M e^{-\beta|t-s|}ds\cr
&&+ \left\| h\right\|_{C_b(\mathbb{R},X)}\left(\int_{|t-s|\geqslant\eta}\varepsilon e^{-\gamma|t-s|}ds + \int_{t-\eta}^{t+\eta} 2(1+H) M ds \right)\cr
&\leqslant& \frac{2(1+H)M\varepsilon}{\beta} + \left\| h\right\|_{C_b(\mathbb{R},X)}\left( \frac{2\varepsilon}{\gamma} + 4(1+H) M \eta \right).
\end{eqnarray}
Therefore, for a given $\epsilon>0$, we can choose a small $\eta>0$, then a small $\varepsilon>0$ and there exists  $l_\epsilon>0$ such that every interval of length $l_\epsilon$ contains at least a number $T$ and the following inequality holds
\begin{equation*}
\left\|S(v)(t+T) - S(v)(t)\right\|_X\leqslant \epsilon
\end{equation*}
for $t\in \mathbb{R}$. This shows that the solution operator maps an almost periodic function to an almost periodic function, i.e., $S:AP(\mathbb{R},X)\to AP(\mathbb{R},X)$. Hence, the existence of almost periodic mild solution is proven. The uniqueness holds clearly.

{\bf Issue 2. Now we prove the assertion of asymptotically almost periodic case.} 
From the existence and uniqueness of mild solution on the half line time-axis otained in Assertion $(i)$, we can define the solution operator $\mathbb{S}:C_b(\mathbb{R}_+,X)\to C_b(\mathbb{R}_+,X)$ associating with the linear equation \eqref{linearAbstract} as follows 
\begin{equation*}\label{SolOp-1}
\mathbb{S}(v)(t) =u(t) = U(t,0)\zeta_0  + \int_0^{+\infty} G(t,s)h(s,v(s))ds \hbox{   where   } \zeta_0 \in X_0 = P(0)X.
\end{equation*}
Since $v\in AAP(\mathbb{R}_+,X)$ and $h(\cdot,v(\cdot))\in AAP(\mathbb{R}_+,X)$, we assume that $h(\cdot,v(\cdot))=h_1(\cdot,v(\cdot))+h_2(\cdot,v(\cdot)) \in AAP(\mathbb{R}_+,X)$, where $h_1(\cdot,v(\cdot))\in AP(\mathbb{R},X)$ and $h_2\in C_0(\mathbb{R}_+,X)$. We will show that $\mathbb{S}(v) \in AAP(\mathbb{R}_+,X)$. Indeed, we can rewrite 
\begin{eqnarray}
\mathbb{S}(v)(t) &=& U(t,0)\zeta_0 + \int_0^{+\infty} G(t,s)h(s,v(s))ds \cr
&=& U(t,0)\zeta_0 + \int_0^{+\infty} G(t,s)[h_1(s,v(s)) + h_2(s,v(s))] ds\cr
&=& \int_{-\infty}^{+\infty} G(t,s)h_1(s,v(s)) ds \cr
&&+ \left( U(t,0)\zeta_0 -\int_{-\infty}^0G(t,s)h_1(s,v(s)) ds + \int_0^\infty G(t,s)h_2(s,v(s)) ds \right)\cr
&=& \mathbb{S}_1(v)(t) + \mathbb{S}_2(v)(t),
\end{eqnarray} 
where
\begin{equation}
\mathbb{S}_1(v)(t) = \int_\mathbb{R} G(t,s)h_1(s,v(s)) ds
\end{equation}
and 
\begin{equation}
\mathbb{S}_2(v)(t) = U(t,0)\zeta_0 -\int_{-\infty}^0G(t,s)h_1(s,v(s)) ds + \int_0^\infty G(t,s)h_2(s,v(s)) ds.
\end{equation}
By the same way as \eqref{AP1} we can estimate
\begin{equation}
\left\|\mathbb{S}_1(v)(t+T) - \mathbb{S}_1(v)(t) \right\|_X \leqslant \frac{2(1+H)M\varepsilon}{\beta} + \left\| h_1 \right\|_{C_b(\r,X)}\left( \frac{2\varepsilon}{\gamma} + 4(1+H) M \eta \right)
\end{equation}
for the funtion $h_1(\cdot,v(\cdot))\in AP(\mathbb{R},X)$ satisfies
\begin{equation}
\left\|h_1(t+T, v(t+T)) - h_1(t,v(t))\right\|_X\leqslant \epsilon.
\end{equation}
This shows that $\mathbb{S}_1(v)\in AP(\r,X)$.

Now, we prove that $\lim\limits_{t\to +\infty}\norm{\mathbb{S}_2(v)(t)}_X=0$. Clearly, the first term of $\mathbb{S}_2(v)$ can be estimated as
\begin{equation}
\|U(t,0)\zeta_0 \|_X \leqslant Ne^{-\beta t}\norm{\zeta_0}_X \longrightarrow 0 \hbox{   as   } t\to +\infty.
\end{equation}
Moreover, we can also estimate the second term of $\mathbb{S}_2(v)$ as 
\begin{eqnarray}
\left\|\int_{-\infty}^0G(t,s)h_1(s,v(s)) ds\right\| &\leqslant& \int_{-\infty}^0 \left\| G(t,s)h_1(s,v(s))\right\|_X ds\cr
&\leqslant& (1+H)N\int_{-\infty}^0 e^{-|t-s|} \| h_1(s,v(s))\|_X ds\cr
&\leqslant& (1+H)N\|h_1\|_{C_b(\mathbb{R},X)} \int_0^{+\infty}e^{-(t+s)}ds\cr
&\leqslant& (1+H)N\|h_1\|_{C_b(\mathbb{R},X)} e^{-t}\cr
&\longrightarrow& 0 
\end{eqnarray}
as $t$ tends to infinity. Therefore, we remain to prove 
\begin{equation}\label{Limit}
\lim_{t\to +\infty}\int_0^\infty \|G(t,s)h_2(s,v(s))\|_X ds =0.
\end{equation}
Indeed, from $\lim\limits_{t\to +\infty}\|h_2(t,v(t))\|_X = 0$, for each $\varepsilon>0$, there exists a constant large enough $t_0>0$ such that for all $t>t_0$, we have
\begin{equation*}
\|h_2(t,v(t))\|_X < \varepsilon.
\end{equation*}
This leads to the following estimates
\begin{align}
&\int_0^\infty \|G(t,s)h_2(s,v(s))\|_X ds\nonumber\\
&\quad\quad\leqslant \int_0^{t_0} \|G(t,s)h_2(s,v(s))\|_X ds + \int_{t_0}^\infty \|G(t,s)h_2(s,v(s))\|_X ds \nonumber\\
&\quad\quad\leqslant \int_0^{t_0} \|G(t,s)h_2(s,v(s))\|_X ds + \int_{t_0}^\infty Ne^{-\beta |t-s|}\|h_2(s,v(s))\|_X ds \nonumber\\
&\quad\quad\leqslant \int_0^{t_0} \|G(t,s)h_2(s,v(s))\|_X ds + \varepsilon\int_{t_0}^\infty Ne^{-\beta |t-s|} ds \nonumber\\
&\quad\quad\leqslant \int_0^{t_0} (1+H)Ne^{-\beta|t-s|}\|h_2(s,v(s))\|_X ds + \varepsilon\int_{t_0}^\infty (1+H)Ne^{-\beta |t-s|} ds \nonumber\\
&\quad\quad\leqslant (1+H)N \left(\frac{\left(-e^{-\beta|t|} + e^{-\beta|t-t_0|}\right) \|h_2\|_{C_b(\mathbb{R}_+,X)}}{\beta} +  \frac{2(1-e^{-\beta |t-t_0|})\varepsilon}{\beta} \right)\nonumber\\
&\quad\quad\leqslant \epsilon\nonumber
\end{align}
which holds for each $\epsilon>0$, provided that $t\gg t_0$ large enough. This shows the limit \eqref{Limit}. Our proof is complete for the case of asymptotically almost periodic functions.

{\bf Issue 3. Finally, we prove the assertion for pseudo almost periodic case.}
{Since a pseudo almost periodic function is determined on the whole line time-axis, we consider the solution operator $S: C_b(\mathbb{R},X)\to C_b(\mathbb{R},X)$ similarly as in Issue 1, i.e., {formulated} by \eqref{SolOpe1}. We will prove that this solution operator maps $PAP(\mathbb{R},X)$ into itself, then the linear equation \eqref{linearAbstract} has a unique PAP-mild solution.}

Since $h(\cdot,v(\cdot)) \in PAP(\mathbb{R},X)$, we have the followin decomposition 
$$h(s,v(s)) = g(s) + \varphi(s),\quad\text{where } g\in AP(\mathbb{R},X) \text{ and } \varphi\in PAP_0(\mathbb{R},X).$$
Then, we can rewrite $S(h)(t)$ as follows 
\begin{eqnarray}\label{solform-1}
S(v)(t) &=& \int_{\mathbb{R}} G(t, s) g(s) {d} s+\int_{\mathbb{R}} G(t, s) \varphi(s) {d} s \cr
&=& \mathcal{S}(g)(t) + \mathcal{S}(\varphi)(t),
\end{eqnarray}
where
\begin{equation}
\mathcal{S}(g)(t) = \int_{\mathbb{R}} G(t, s) g(s) {d} s \text{   and   } \mathcal{S}(\varphi)(t) = \int_{\mathbb{R}} G(t, s) \varphi(s) {d} s.
\end{equation}

Similar to Issue 1, we get that {the first function $\mathcal{S}(g)$ in \eqref{solform-1} is an almost periodic function with respect to the time variable}. In order to show that $S(v)\in PAP(\mathbb{R},X)$, {we remains to verify that the second funtion (in \eqref{solform-1}) $\mathcal{S}(\varphi)$ belongs to $PAP_0(X)$. Setting 
\begin{equation*}
I(r) = \frac{1}{2 r} \int_{-r}^r\|\mathcal{S}(\varphi)(t)\| {d} t.
\end{equation*}
We need to prove that $\lim\limits_{r\to \infty}I(r)=0$.} Indeed, we estimate that
\begin{equation*}
\begin{aligned}
I(r) &=   \frac{1}{2 r} \int_{-r}^r\left\|\int_{-\infty}^{\infty} G(t, s) \varphi(s) \mathrm{d} s\right\|_X {d} t \cr
& \leqslant  \frac{1}{2 r} \int_{-r}^r {d} t \int_{-\infty}^{+\infty}\|G(t, s) \varphi(s)\|_X {d} s \cr
&\leqslant  \frac{1}{2 r} \int_{-r}^r {d} t \int_{-\infty}^{+\infty} (1+H)N e^{-\beta|t-s|}\|\varphi(s) \|_X {d} s \hbox{      (we used \eqref{ineq-GF})}.
\end{aligned}
\end{equation*}
{By Fubini's theorem with noting that the functions $s\mapsto e^{-\beta|t-s|}\|\varphi(s)\|_X$ and $t\mapsto e^{-\beta|t-s|}\|\varphi(s)\|_X$ are integrable and by changing variable yields}
\begin{align}\label{I(r)}
I(r) &\leqslant \dfrac{(1+H)N}{2r}\int_{-r}^r {d}t\int_{-\infty}^{+\infty} e^{-\beta|s|}\|\varphi(t-s) \|_X {d} s\cr
&=(1+H)N\int_{-\infty}^{+\infty} e^{-\beta|s|}\left(\dfrac{1}{2r}\int_{-r}^r\| \varphi(t-s) \|_X  dt\right) {d} s\cr
&=(1+H)N\int_{-\infty}^{+\infty} e^{-\beta|s|}\left(\dfrac{1}{2r}\int_{-r-s}^{r-s}\|\varphi(t) \|_X dt\right) {d} s.
\end{align}
{We have clearly that}
\begin{equation*}
e^{-\beta|s|}\dfrac{1}{2r}\int_{-r-s}^{r-s}\|\varphi(t) \|_X dt \leqslant e^{-\beta|s|}\|\varphi\|_{C_b(\mathbb{R},X)} = \gamma(s).
\end{equation*}
{Since $\gamma(s)$ is integrable on $(-\infty,+\infty)$, we use the Lebesgue dominated convergence to obtain from \eqref{I(r)} that} 
\begin{eqnarray*}
0\leqslant \lim_{r\to\infty} I(r) &\leqslant& (1+H)N\lim_{r\to \infty} \int_{-\infty}^{+\infty} e^{-\beta|s|}\left(\dfrac{1}{2r}\int_{-r-s}^{r-s}\|\varphi(t) \|_X dt\right) {d} s\cr
&=&(1+H)N \int_{-\infty}^{+\infty}e^{-\beta|s|}\lim_{r\to \infty} \left(\dfrac{1}{2r}\int_{-r-s}^{r-s}\|\varphi(t) \|_X dt\right) {d} s\cr
&=& (1+H)N \int_{-\infty}^{+\infty}e^{-\beta|s|}\times 0 \, ds \hbox{       (we used $\varphi\in PAA_0(\mathbb{R}_+,X)$)}\cr
&=& 0.
\end{eqnarray*} 
This shows that $\lim\limits_{r\to \infty}I(r)=0$ and the proof is complete.
\end{proof}

\section{The semi-linear equation: well-posedness and exponential stability}\label{S4}
\subsection{Well-posedness of AP-, AAP- and PAP-mild solutions}\label{S41}
{In this section, based on the results of Theorem \ref{lem2}, we establish the well-posedness of AP-, AAP-, PAP-mild solutions and exponential stability for the following semi-linear evolution equation}
\begin{equation}\label{eq:21}
\begin{cases}
\dfrac{du}{dt}=A(t)u(t)+h(t,u(t)) \\
u(s)=\phi \in X.
\end{cases}
\end{equation}

For $\rho>0$ we denote by $B^Z_{\rho}$ the ball centered at zero and radius $\rho$ in $Z$, where $Z$ can be   $AP(\mathbb{R},X)$, $AAP(\mathbb{R}_+,X)$ or $PAP(\mathbb{R},X)$, i.e, 
$$B^Z_\rho:=\{x\in Z: \|x\|_{C_b(\mathbb{R},X)} \leqslant \rho\}.$$
{Note that, for the case of $AAP(\mathbb{R}_+,X)$ we need the condition $\|x\|_{C_b(\mathbb{R}_+,X)}\leqslant \rho$ in $B_\rho^Z$. However, by the similarity of proofs we denote the ball in the same way.}

Moreover, we add the following conditions on the nonlinear term.
\begin{assumption}\label{nonlinear}
The function $h:\mathbb{R} \times X\rightarrow X$ satisfies
 \begin{itemize}
\item[$(i)$] For $u\in Z$, $h(\cdot,u(\cdot))\in Z$, 
\item[$(ii)$]  $\|h(\cdot,0)\|_{C_b(\mathbb{R},X)}\leqslant\gamma$, where $\gamma$ is a non-negative constant,
\item[$(iii)$] There exist  positive constants  $\rho$ and $L$ such that
$$ \|h(t, v_1(t))-h(t, v_2(t))\|_{X}\leqslant L\|v_1(t)-v_2(t)\|_{X}$$
for each $t\in \mathbb{R}$ and all $v_1, v_2\in C_b(\mathbb{R},X) \hbox{ with }\|v_1\|_{C_b(\mathbb{R},X)} \leqslant \rho$ and  $\|v_2\|_{C_b(\mathbb{R},X)} \leqslant \rho.$
\end{itemize}
\end{assumption}
 
From the proof of Theorem \ref{lem2}, we define the mild solution to semi-linear equation \eqref{eq:21} on the whole-line time axis by the function $u$ satisfying the following integral equation
\begin{equation}\label{eq:23}
u(t)=U(t,s)u(s)+\int\limits^{t}_{s}G(t,\tau)h(\tau,u(\tau))d\tau = \int_\mathbb{R}G(t,\tau)h(\tau,u(\tau))d\tau.
\end{equation}
{If we consider the Cauchy problem with initial data $u(0)=u_0\in X$ of equation \eqref{nonlinear}, then we can define
\begin{equation*}\label{eq:23'}
u(t)=U(t,0)\zeta_0 + \int\limits^{+\infty}_{0}G(t,\tau)h(\tau,u(\tau))d\tau,
\end{equation*}
where $\zeta_0 = u_0 - \int_0^\infty G(0,s)h(s,u(s))ds \in P(0)X$.}

{We state and prove the existence and uniqueness of the AP-, AAP- and PAP- mild solution to equation (\ref{eq:21}) in the following theorem.}
\begin{theorem}\label{th:2}
Assume that the function $h$ satisfies Assumption \ref{nonlinear}. Then, the following assertions hold
\begin{itemize}
\item[$(i)$]
If $L$ and $\gamma$ are small enough, semi-linear equation (\ref{eq:21}) has one and only one  AP- (resp. PAP-) mild solution $\hat {u}$ on a small ball of $Z$.
\item[$(ii)$]
If $\|u_0\|_X$, $L$ and $\gamma$ are small enough, semi-linear equation (\ref{eq:21}) has one and only one  AAP- mild solution $\hat {u}$ on a small ball of $Z$.
\end{itemize}

\end{theorem}
\begin{proof}
$(i)$ We provide the detailed proof for AP- and PAP- mild solutions which are defined on the whole-line time axis. The case of AAP- mild solution defined on the half-line time axis is done in the same way with noting the intial data.
Let $v$ be a function in $B_{\rho}^Z$. Consider the integral equation
\begin{equation}\label{eq:24}
u(t)= \int_\mathbb{R} G(t,\tau)h(\tau,v(\tau)d\tau.
\end{equation}
{Applying Theorem \ref{lem2}, there exists a unique mild solution $u \in Z$ to equation (\ref{eq:24}). Therefore, we can set a map $\Phi: Z\to Z$ by $\Phi(v)(t)=u(t)$ which is solution of equation \eqref{eq:24}.}

In order to establish the well-posedness of equation \eqref{eq:23}, we prove that if $L$ and $\rho$ are small enough, then the map $\Phi$ acts from $B^Z_{\rho}$ into itself and is a contraction. 

We have
\begin{align*}\label{eq:25}
\|h(t, v(t))\|_{X} &\leqslant \|h(t, v(t))-h(t,0)\|_{X}+\|h(t,0)\|_{X}\cr
&\leqslant  L\|v\|_{C_b(\mathbb{R},X)}+\gamma\cr
&\leqslant L\rho+\gamma
\end{align*}
for all $t\in \mathbb{R}$.

Applying the boundedness of mild solution in the proof of Theorem \ref{lem2}, we obtain that for
$v\in B_{\rho}$ there exists a unique AP- (resp. PAP-) mild solution $u$ to (\ref{eq:24}) satisfying
\begin{eqnarray}\label{eq:26}
 \|u\|_{C_b(\mathbb{R},X)}&\leqslant& \frac{2(1+H)N}{\beta} \|h(\cdot, v(\cdot)) \|_{C_b(\mathbb{R},X)}\cr
 &\leqslant& \frac{2(1+H)N}{\beta}(L\rho+\gamma)\cr
 &\leqslant& \rho
\end{eqnarray}
provided that $L$ and $\gamma$ small enough.
Therefore, with these values of $L$ and $\gamma$ the map $\Phi$ acts from $B_{\rho}$ into itself. 

Furthermore, for $v_1,v_2\in B_{\rho}$ 
and $u_1=\Phi(v_1)$, $u_2=\Phi(v_2)$ by the representation (\ref{eq:24}) we
obtain that $u=\Phi(v_1)-\Phi(v_2)$ is the unique AP- (resp. PAP-) mild solution to the integral equation
\begin{equation*}
u(t) = \int_{\mathbb{R}} G(t,\tau)(h(\tau, v_1(\tau)-h(\tau, v_2(\tau))d\tau \hbox{\ for all }t.
\end{equation*}
Similar as above, we can estimate
\begin{align*}\label{eq:260}
 \|\Phi(v_1)-\Phi(v_2)\|_{C_b(\mathbb{R},X)}&\leqslant \frac{2(1+H)N}{\beta} \|h(\cdot, v_1(\cdot))-h(\cdot, v_2(\cdot)) \|_{C_b(\mathbb{R},X)}\cr
&\leqslant \frac{2L(1+H)N}{\beta}\|v_1-v_2\|_{C_b(\mathbb{R},X)}.
\end{align*}
We thus obtain that if $L$ is small enough, then $\Phi$ is a contraction.

By fixed point arguments, for the above values of $L$ and $\gamma$ there exists a unique fixed point $\hat{u}$ of $B^Z_{\rho}$, and by definition of $\Phi$, this function $\hat {u}$ is the unique AP- (resp. PAP-) mild solution to integral equation \eqref{eq:24}, then semi-linear equation \eqref{eq:21}. The proof of assertion $(i)$ is complete.

$(ii)$ The well-posedness of AAP- mild solution for semi-linear equation \eqref{abstract} is done as the similar above argument with inequality \eqref{eq:26} replaced by 
\begin{eqnarray*}
\|u\|_{C_b(\mathbb{R}_+,X)} &\leqslant& N\|u_0\|_X + \frac{3(1+H)N}{\beta}\|h(\cdot,v(\cdot))\|_{C_b(\mathbb{R}_+,X)}\cr
&\leqslant& N\|u_0\|_X + \frac{3(1+H)N}{\beta}(L\rho+\gamma).
\end{eqnarray*}
Therefore, we need the small enough conditions of $\|u_0\|_X$, $L$ and $\gamma$ to guarantees the well-posedness. The details are left to the readers. The proof of Assertion $(ii)$ is complete.
\end{proof}

\subsection{Exponential stability}
In this subsection we state and prove the exponential stability of AP-, AAP- and PAP-mild solutions of equation \eqref{abstract} obtained in Theorem \ref{th:2}.
\begin{theorem}\label{stable}
Let the assumption of Theorem \ref{th:2} be satisfied with positive constants $L,\,\gamma,\,\rho$. Denote by $\hat{u}$ the unique mild solution in $B_\rho^Z(0)$ defined by Theorem \ref{th:2}. Assume further that for all functions  $x_1,x_2 \in B_{2\rho}(0)$ (where $B_{2\rho}(0)$ denotes the ball centered at zero and radius $2\rho$ in $C_b(\mathbb{R}_+,X)$), there is a positive constant $L_1$ for which 
\begin{equation*}\label{Lip2}
\|h(t,x_1(t))-h(t,x_2(t))\|_{X}\leqslant L_1\|x_1(t)-x_2(t)\|_{X}
\end{equation*}
for each $t \in \mathbb{R}_+$.

If $L_1$ is small enough, then for each $v_0\in B^X_{\frac{\rho}{2N}}(P(0)\hat{u}(0))\cap P(0)X$ (where $B^X_{\frac{\rho}{2N}}(P(0)\hat{u}(0))$ denotes the ball centered at $P(0)\hat{u}(0)$ and radius $\frac{\rho}{2N}$ in $X$), there exists a unique mild solution $u(t)$ of equation \eqref{eq:21} in $B_\rho^Z(\hat{u})$ such that $P(0)u(0)=v_0$. Moreover, the following estimate holds for $u(t)$ and $\hat{u}(t)$
\begin{equation*}
\label{stable1}
\|u(t)-\hat{u}(t)\|_X \leqslant C e^{-\mu t}\|P(0) u(0)-P(0) \hat{u}(0)\|_X
\end{equation*}
for all $t>0$. Here, the positive constants $C$ and $\mu$ are independent of $u$ and $\hat{u}$.
\end{theorem}
\begin{proof}
In the theorem we consider $t>0$, then the proof for AP- and PAP- mild solutions (which are defined on the whole-line time axis) and the one for AAP- mild solution (defining on the half-line time axis) are similar since we can rewrite the mild solution $u(t)$ for each $t>0$ as follows
\begin{eqnarray}\label{solre}
u(t) &=& \int_{\mathbb{R}} G(t,s)h(s,u(s))ds \cr
&=& \int_{-\infty}^0G(t,s)h(s,u(s))ds + \int_0^{+\infty} G(t,s)h(s,u(s))ds\cr
&=& U(t,0)\int_{-\infty}^0G(0,s)h(s,u(s))ds + \int_0^{+\infty} G(t,s)h(s,u(s))ds\cr
&=& U(t,0)\left(u(0) - \int_0^{+\infty} G(0,s)h(s,u(s))ds \right) + \int_0^{+\infty}G(t,s)h(s,u(s))ds\cr
&=& U(t,0)\zeta_0 + \int_0^{+\infty}G(t,s)h(s,u(s))ds,
\end{eqnarray}
where $\zeta_0 = u(0) - \int_0^{+\infty} G(0,s)h(s,u(s))ds  \in P(0)X$. 

First, we prove the existence of solution $u$. For given $v_0 \in B^X_{\frac{\rho}{2 N}}(P(0) \hat{u}(0)) \cap P(0) X$ and $\omega\in Z$, we define the transformation $F:Z\to Z$ by
$$
(F w)(t)=U(t, 0) v_0+\int_0^{+\infty} G(t, s)h(s,w(s)) ds.
$$
Note that, this definition is well-defined because of the well-posedness of linear integral equation obtained in Theorem \ref{rem}.
We will prove that $F$ acts from $B^Z_\rho(\hat{u})$ into itself and is a contraction. Combining with fixed point arguments we have the existence of solution $u$.

Let $w \in B^Z_\rho(\hat{u})$. We have
$$
\|w\|_{C_b(\mathbb{R},X)} \leqslant\|w-\hat{u}\|_{C_b(\mathbb{R},X)}+\|\hat{u}\|_{C_b(\mathbb{R},X)} \leqslant 2 \rho.
$$
It is easy to see that
$$\hat{u}(0) - \int_0^{+\infty}G(0,s)h(s,\hat{u}(s))ds = P(0)\hat{u}(0).$$
Thus, we get from \eqref{solre} that
$$\hat{u}(t)=U(t,0)P(0)\hat{u}(0) + \int_0^{+\infty}G(t,s)h(s,\hat{u}(s))ds.$$
Hence, by setting
$$
y(t)=(F w)(t)=U(t, 0) v_0+\int_0^{+\infty} G(t,s)h(s,w(s)) ds,
$$
we obtain from Assumption \ref{dich} and inequality \eqref{ineq-GF} that
\begin{align*}
\|y(t)-\hat{u}(t)\|_X & \leqslant N e^{-\beta t}\left\|v_0-P(0) \hat{u}(0)\right\|_X\cr
&\quad+(1+H) N \int_0^{+\infty} e^{-\beta|t-s|} \|h(s,w(s))-h(s,\hat{u}(s))\|_{X} ds\cr
& \leqslant N\left\|v_0-P(0) \hat{u}(0)\right\|_X+\frac{2(1+H) N L_1 \rho}{\beta}
\end{align*}
for all $t \geqslant 0$. This together with the fact that $\left\|v_0-P(0) \hat{u}(0)\right\|_X \leqslant \dfrac{\rho}{2 N}$ yield
\begin{align*}
\|F w-\hat{u}\|_{C_b(\mathbb{R},X)} &\leqslant N\left\|v_0-P(0) \hat{u}(0)\right\|_X+\frac{2(1+H) N L_1 \rho}{\beta}\cr
&\leqslant \dfrac{\rho}{2} + \frac{2(1+H) N L_1 \rho}{\beta}.
\end{align*}
Therefore, if $L_1$ is small enough satisfying $\dfrac{2(1+H) N L_1}{\beta}<\dfrac{1}{2}$, then the transformation $F$ acts from $B^Z_\rho(\hat{u})$ into itself.

Let $x, z \in B^Z_\rho(\hat{u})$, we have $x,z\in B^Z_{2\rho}(0)$. Then, assumption \eqref{Lip2} implies that
\begin{align*}
\|(F x)(t)-(F z)(t)\|_X & \leqslant \int_0^{+\infty}\|G(t,s)\| \| h(s,x(s))-h(s,z(s)) \|_X ds \cr
& \leqslant(1+H) N L_1\int_0^{+\infty} e^{-\beta|t-s|} \|x(s)-z(s)\|_{X} ds.
\end{align*}
Therefore,
\begin{equation*}
\|F x-F z\|_{C_b(\mathbb{R},X)} \leqslant \frac{2(1+H) N L_1}{\beta}\|x-z\|_{C_b(\mathbb{R},X)}.
\end{equation*}
Since $\dfrac{2(1+H) N L_1}{\beta}<\dfrac{1}{2}$, we obtain that $F: B^Z_\rho(\hat{u}) \rightarrow B^Z_\rho(\hat{u})$ is a contraction. This guarantees the existence of a unique $u \in B^Z_\rho(\hat{u})$ such that $F u=u$. By the definition of $F$, we imply that $u$ is the unique solution in $B^Z_\rho(\hat{u})$ of equations \eqref{eq:21}.

Finally, we prove the stable inequality \eqref{stable1}. By formula \eqref{solre}, we can write
\begin{equation*}
u(t)-\hat{u}(t)=U(t, 0)(P(0) u(0)-P(0) \hat{u}(0))+\int_0^{+\infty} G(t,s)(h(s,u(s))-h(s,\hat{u}(s))) ds.
\end{equation*}
Since $u,\hat{u}\in B^Z_\rho(0)$, it follows from Assumption \ref{nonlinear} and \eqref{dich} that
\begin{eqnarray*}
\|u(t)-\hat{u}(t)\|_X & \leqslant& N e^{-\beta t}\|P(0) u(0)-P(0) \hat{u}(0)\|_X\cr
&&+(1+H) N \int_0^{+\infty} e^{-\beta|t-s|}\|h(s,u(s))-h(s,\hat{u}(s))\|_X ds \cr
& \leqslant& N e^{-\beta t}\|P(0) u(0)-P(0) \hat{u}(0)\|_X\cr
&&+(1+H) N L_1 \int_0^{+\infty} e^{-\beta|t-s|}\|u(s)-\hat{u}(s)\|_X ds.
\end{eqnarray*}
Applying now a Gronwall-typed inequality \cite[Corollary III.2.3]{Daleckii},  we obtain for $\gamma=(1+H) N L_1<\dfrac{\beta}{2}$ that
\begin{equation*}
\|u(t)-\hat{u}(t)\|_X \leqslant C e^{-\mu t}\|P(0) u(0)-P(0) \hat{u}(0)\|_X 
\end{equation*}
for  $\mu:=\sqrt{\beta^2-2 \gamma \beta},\; C:=\frac{2 N \beta}{\beta+\sqrt{\beta^2-2 \gamma \beta}}
$. The proof is complete.
\end{proof}

\section{Constructions of Examples and Applications}\label{S5}
\subsection{The scalar Li\'enard equation}\label{APP}
In this subsection we comeback to establish the well-posedness of AP-, AAP and PAP- mild solutions and their stability of the scalar Li\'enard equation by applying abstract results obtained in Section \ref{S3} and Subsection \ref{S41}. {Note that, the original scalar Li\'enard equation \eqref{eq:2.2matrix} is in fact linear evolution equation under form \eqref{linearAbstract}. Therefore, we can use Theorem \ref{lem2} to obtain the well-posedness of equation \eqref{eq:2.2matrix}.}
However, in the wide of application of abstract results we consider the following scalar Li\'enard equation (which is in fact semilinear equation):
\begin{equation}\label{LienardEq}
\frac{d}{dt}\begin{bmatrix}
x\\ 
y
\end{bmatrix}
=A(t)\begin{bmatrix}
 x\\
y
\end{bmatrix} + \begin{bmatrix}
0\\ {-g(t,0)} + e(t,x)
\end{bmatrix},
\end{equation}
with $f(x(t))$ replaced by $f(t,x(t))$ and $e(t)$ replaced by $e(t,x(t))$ and
$$A(t) = \int_0^1
\begin{bmatrix}
-f(t,s{\alpha}(t)) &1\\
-g_x(t,s{\alpha(t)}) &0
\end{bmatrix}
ds.$$
We consider $f(t,x) = q(t)$ and $g(t,x) =cx(t) + r(t)x(t)$, where $q$ and $r$ are almost periodic functions and $c$ is a suitable constant. Then, we have
\begin{equation*}
A(t) = \int_0^{\alpha(t)} \begin{bmatrix}
-q(t) &1\\
-c-r(t) &0
\end{bmatrix}
ds = \begin{bmatrix}
-q(t)\alpha(t) &1\\
-(c+r(t))\alpha(t) &0
\end{bmatrix}.
\end{equation*}
Therefore, the acting of $A(t)$ on a vector $\begin{bmatrix}
x\\
y
\end{bmatrix}$ can be rewritten as
\begin{equation*}
A(t)\begin{bmatrix}
x\\
y
\end{bmatrix}= \begin{bmatrix}
-q(t)\alpha(t) &1\\
-(c+r(t))\alpha(t) &0
\end{bmatrix}\begin{bmatrix}
x\\ y
\end{bmatrix}=
\begin{bmatrix}
-q(t) &1\\
-(c+r(t)) &0
\end{bmatrix} \begin{bmatrix}
x\\
y
\end{bmatrix}.  
\end{equation*}
Since the functions $q$ and $r$ are almost periodic, we have $A(t) = \begin{bmatrix}
-q(t) &1\\
-(c+r(t)) &0
\end{bmatrix}$ is almost periodic. 

The following lemma (see \cite[Lemma 1]{Zhu}) provides the condition of functions $f$ and $g$ which guarantees that there is a evolution family $\left\{ U(t,s)\right\}_{t\geqslant s}$ satisfied Assumption \ref{dich} that is associated with the linear equation 
\begin{equation}\label{LienardEq1}
\frac{d}{dt}\begin{bmatrix}
x\\ 
y
\end{bmatrix}
=A(t)\begin{bmatrix}
 x\\
y
\end{bmatrix}.
\end{equation}
\begin{lemma}\label{lem1} 
Assume that there exist positive constants $M, \delta,$ such that for any $x\in \mathbb{R}$ satisfying
\begin{itemize}
\item[$(i)$] $|f(t,x)| + |g_x(t,x)|\leqslant M;$
\item[$(ii)$] $4g_x(t,x)+f^2(x)<-\delta;$
\end{itemize}
then there is a evolution family $\left\{ U(t,s)\right\}_{t\geqslant s}$ satisfied Assumption \ref{dich} that is associated with the linear equation \eqref{LienardEq1}.
\end{lemma}
\begin{proof}
The proof was given in \cite{Zhu}.
\end{proof}
To guarantee the conditions in above lemma we can choose 
\begin{equation}\label{Ffunc}
f(t,x)= q(t)  = \sin t + \sin(\sqrt 2 t)\hbox{    for    } \sup_{t\in \mathbb{R}}|q(t)|< 2.
\end{equation}
and   
\begin{equation}\label{Gfunc}
g(t,x) = -2x(t) + r(t)x(t),
\end{equation}
where $r(t)=\cos t + \cos \sqrt{2} t$ for    $\sup\limits_{t\in \mathbb{R}} |r(t)| < 2$.
We can verify easly the condition in Lemma \ref{lem1} as 
\begin{eqnarray*}
&&|f(x)| + |g_x(t,x)| = |q(t)| + \left| -2 + r(t)\right| <6,\cr
&& 4g_x(t,x) + f^2(x) = -8 + r(t) + |q(t)|^2 < -2.
\end{eqnarray*}
{From the above construction, we have the existence of evolution family $\left\{ U(t,s)\right\}_{t\geqslant s}$ satisfying  exponential dichtonomy in Assumption \ref{dich}. Moreover, since the operator $A(t)$ is almost periodic, the Green function $G(t,s)$ associating with $U(t,s)$ satisfies almost periodic property in Assumption \ref{GreenAss}.}

In view of definition of mild solutions in Subection \ref{S41}, The mild solution of equation \eqref{LienardEq} on the whole line time-axis is defined by
\begin{eqnarray}\label{AppMild1}
\begin{bmatrix}
x\\y
\end{bmatrix}(t) &=& U(t,s)u(s) + \int_s^t U(t,s)\begin{bmatrix}
0\\x(s)+e(s)
\end{bmatrix}ds \cr
& = & \int_{\mathbb{R}} G(t,s)\begin{bmatrix}
0\\g(s,0) +e(s,x(s))
\end{bmatrix}ds\cr
& = & \int_{\mathbb{R}} G(t,s)\begin{bmatrix}
0\\e(s,x(s))
\end{bmatrix}ds \hbox{     (because $g(s,0)=0$)},
\end{eqnarray}
where $G(t,s)$ is Green function associating with $U(t,s)$.

On the other hand, if we consider Cauchy problem of equation \eqref{LienardEq} with the initial data $\begin{bmatrix}
x(0)\\y(0)
\end{bmatrix}=\begin{bmatrix}
x_0\\y_0
\end{bmatrix}$, then the mild solution (on the half line time-axis) is defined by
\begin{eqnarray}\label{AppMild2}
\begin{bmatrix}
x\\y
\end{bmatrix}(t) &=& U(t,0)\begin{bmatrix}x_0\\y_0 \end{bmatrix} + \int_0^t U(t,s)\begin{bmatrix}
0\\ g(s,0) + e(s,x(s))
\end{bmatrix}ds \cr
&=& \zeta_0 + \int_0^\infty G(t,s) \begin{bmatrix}
0\\ e(s,x(s))
\end{bmatrix}ds,
\end{eqnarray}
where $\zeta_0 = \begin{bmatrix}x_0\\y_0 \end{bmatrix} - {\int}_0^\infty G(0,s)\begin{bmatrix}
0\\ e(s,x(s))
\end{bmatrix}ds$.

To guarantees the conditions in Assumption \ref{nonlinear} are valid for the function $h(s)= \begin{bmatrix} 0\\  e(s,x(s))\end{bmatrix}$, we assume that $e(s,x(s))$ satisfies Assumption \ref{nonlinear}. It is easly to choose such functions, for example:  $e(s,x(s)) = |x(s)|^{m-1} x(s) + k(s)$ for $m\geqslant 2$ and $k\in C_b(\mathbb{R},\mathbb{R})$.

With the above functions $f$, $g$ and $e$ on hand, we apply abstract results obtained in Theorem \ref{th:2} and Theorem \ref{stable} to obtain the results for scalar Li\'enard equation \eqref{LienardEq} in the following theorem. 
\begin{theorem}\label{Appl1}
The following assertions hold
\begin{itemize}
\item[$(i)$]
Let $f$ and $g$ be functions as \eqref{Ffunc} and \eqref{Gfunc}, respectively. If the Lipschitz constants and $|k(0)|$ are small enough and $k \in C_b(\mathbb{R},\mathbb{R})$ be a given AP- (resp. PAP-) function satifying Assumption \ref{nonlinear}, semilinear equation (\ref{LienardEq}) has one and only one  AP- (resp. PAP-) mild solution $\hat {u}$ satisfying \eqref{AppMild1} on a small ball of $AP(\mathbb{R},\mathbb{R})$ (resp. $PAP(\mathbb{R},\mathbb{R})$). Moreover, these mild solutions are exponential stable in the sense of Theorem \ref{stable}.
\item[$(ii)$]
Let $f$ and $g$ be functions as \eqref{Ffunc} and \eqref{Gfunc}, respectively and $k \in C_b(\mathbb{R}_+,\mathbb{R})$ be a given AAP- function satisfying Assumption \ref{nonlinear}. If the Lipschitz constants, $|u_0|$ and $|k(0)|$, are small enough, semilinear equation (\ref{LienardEq}) has one and only one  AAP- mild solution $\hat {u}$ satisfying \eqref{AppMild2} on a small ball of $AAP(\mathbb{R}_+,\mathbb{R})$. Moreover, the mild solution is exponential stable in the sense of Theorem \ref{stable}.
\end{itemize}
\end{theorem}

Another way to apply the abstract results and give an example for abstract results as follows (see \cite{Xu}) by setting
\begin{eqnarray*}
&&y(t) = x'(t) + a(t)x(t) - \theta_1(t),\cr
&&\theta_2(t) = a(t)\theta_1(t)-\theta_1'(t) + e(t,x(t)).
\end{eqnarray*}
Then, the scalar Li\'enard equation
\begin{equation*}
x'' + f(x)x' + g(x) = e(t,x)
\end{equation*}
is equivalent to the following system
\begin{equation}\label{Re-Formula}
\begin{cases}
    x'(t) = -a(t)x(t) + y(t) + \theta_1(t),\\
     y'(t)= -a^2(t)x(t) + a(t)y(t) - f(x(t))[y(t)-a(t)x(t)+\theta_1(t)] + g(t,x(t)) + \theta_2(t).
 \end{cases}
\end{equation}
In matrix form, we can rewrite \eqref{Re-Formula} as
\begin{equation*}
\frac{d}{dt}\begin{bmatrix}
x\\y
\end{bmatrix}(t) = \mathbb{A}(t)\begin{bmatrix}
x\\y
\end{bmatrix} (t)+ \mathbb{H}\left(t, \begin{bmatrix}x\\y\end{bmatrix}\right),
\end{equation*}
where 
$$\mathbb{A}(t) = \begin{bmatrix}
-a(t)&&0\\0&&a(t)
\end{bmatrix},$$
$$\mathbb{H}\left(\begin{bmatrix}x\\y\end{bmatrix}\right) = \begin{bmatrix}
y +\theta_1\\
-a^2x -f(x)[y-ax+\theta_1] + g(x) + \theta_2
\end{bmatrix}.$$
If we consider that $a(t) \in AP(\mathbb{R},\mathbb{R})$ and $a(t)>\delta>0$ for a fixed $\delta$ and for all $t\in \mathbb{R}$, then there is an evolution family $\left\{ \mathbb{U}(t,s)\right\}_{t\geqslant s}$ associating with the linear evolution equation 
\begin{equation*}
\frac{d}{dt}\begin{bmatrix}
x\\y
\end{bmatrix}(t) = \mathbb{A}(t)\begin{bmatrix}
x\\y
\end{bmatrix} (t).
\end{equation*} 
and $\mathbb{U}(t,s)$ satisfies exponential dichtonomy in Assumption \ref{dich} (see \cite[Lemma 2.4]{Xu} and more details in \cite{Hale}). Since $\mathbb{A}(\cdot)$ is almost periodic, the Green function $\mathbb{G}(t,s)$ satisfies Assumption \ref{GreenAss}.
To verify the conditions in Assumption \ref{nonlinear}, we need to assume that $f \in C_b(\mathbb{R},\mathbb{R})$ is bounded from above by a positive constant $C$ and the functions $g$, $e\in C_b(\mathbb{R},\mathbb{R})$ satisfy Assumption \ref{nonlinear}. Finally, we need that the functions $g$, $\theta_1$ and $e(t,x(t))$ (hence $\theta_2$) are AP- or AAP- or PAP- to guarantee Theorem \ref{Appl1}. The detailed choices of these functions left to the readers.
\begin{remark}
Our abstract results (Theorem \ref{lem2}, Theorem \ref{th:2} and Theorem \ref{stable}) can be applied to other Li\'enard-type equations such as equations with delays (see for example \cite{Gao,Xu}) or for other second-order differential equations such as wave-type equations (see for example \cite{Diagana2011}).
\end{remark}

\subsection{A nonautonom parabolic evolution equation}\label{S52}
To end this paper, we provide an example of nonautonom parabolic equations in which we can apply the abstract results obained in Thereom \ref{th:2} and Theorem \ref{stable}. In particular, we consider the following problem
\begin{equation}\label{ex:2}
{\begin{cases}
\dfrac{\partial w(t,x)}{\partial t}
=a(t)\left[\dfrac{\partial^2w(t,x)}{\partial x^2}+\delta w(t,x)\right]+|w|^{k-1}w(t,x)+g(t,x),\\
\hfill\text { for } 0<x<\pi,\cr
w(t,0)=w(t,\pi)=0.
\end{cases}}
\end{equation}
Here, $\delta\in \mathbb{R}$ and $\delta \neq n^2$  for all $n\in \mathbb{N}$; the function $a(t)\in L^1_{loc}(\mathbb{R})$ is almost periodic (respect to $t$)  and satisfies the condition $0<\gamma_0\leqslant a(t)\leqslant \gamma_1$ for  fixed $\gamma_0,\gamma_1$; the exponent $k\in \mathbb{N}$, $k\geqslant 2$; the function  $g:\mathbb{R}\times [0,\pi] \to \mathbb{R}$ is continuous on $\mathbb{R}\times [0,\pi]$ and almost periodic with respect to $t$.
\def\T{\mathbb{T}}

Setting $X:=L^2[0,\pi]$, and let $A: X\supset D(A)\rightarrow X$ be defined by $Ay=y''+\delta y$, with the domain 
$$ D(A)=\{y\in X: y \text { and }y' \text { are absolutely continuous, } y'\in X, y(0)=y(\pi)=0\}.$$
Observe that $A$ is the generator of an analytic semigroup $\left\{\T(t)\right\}_{t\in \mathbb{R}.}$

By using the spectral mapping theorem for analytic semigroups and the fact that $\sigma(A)=\{-n^2+\delta: n=1,2,3,...\}$, we have  
\begin{equation}\label{spT}
\sigma(\T(t))=e^{t\sigma (A)}=\{e^{t(-n^2+\delta)}: n=1,2,3,...\}
\end{equation}
and hence 
\begin{equation}\label{spTT}
\sigma(\T(t))\cap {\Gamma} = \emptyset\;\forall\; t,
\end{equation}
where ${\Gamma}:=\{\lambda\in \mathbb{C} : |\lambda|=1\}$.

Setting $A(t):=a(t)A$, we obtain that $A(t)$ is the almost periodic (since $a(t)$ is almost periodic).
Moreover, the family  $\left\{A(t)\right\}_{t\in \mathbb{R}}$ generates an evolution family $\left\{U(t,s )\right\}_{t\ge s}$ which is defined by the formula $ U(t,s)=\T(\int_s^ta(\tau)d\tau).$
Since $A(t)$ is almost periodic, we have that the Green function $G(t,s)$ associating with $U(t,s)$ satisfies Assumption \ref{GreenAss} (see Remark \ref{rem}).

By \eqref{spT} and \eqref{spTT} we have that the analytic semigroup $(\T(t))_{t\in \mathbb{R}}$ is hyperbolic 
(or has an exponential dichotomy) with the projection $P$ satisfying
\begin{enumerate}
\item[$(i)$] $ \|\T(t)x\|\leqslant Ne^{-\beta t}\|x\|$ for $x\in PX;$
\item[$(ii)$] $\|\T(-t)_{|\mathrm{Kerr}P}x\|=\|(\T(t)_{|\mathrm{Kerr}P})^{-1}x\|\leqslant Ne^{-\beta t}\|x\|$ for $x\in \text{Ker}P;$
where the invertible operator $\T(t)_{|\mathrm{Kerr}P}$ is the restriction of $T(t)$ to Ker$P$, and  $N$, $\beta$ are positive constants.
\end{enumerate}
The above estimates of $\left\{\mathbb{T}(t)\right\}_{t\in \mathbb{R}}$ implies clearly that the evolution family $\left\{U(t,s )\right\}_{t\geqslant s}$ has an exponential dichotomy with the projection  $P(t)=P$ for all $t$ and the dichotomy constants $N>$ and $ \nu:=\nu(\beta)>0$ as follows
\begin{equation}
\begin{aligned}
\|U(t,s)x\|&\leqslant Ne^{-\nu (t-s)}\|x\| \text { for } x\in PX,\,t\geqslant s; \\ 
\|U(s,t)_{|\mathrm{Kerr}P} x\|&\leqslant Ne^{-\nu(t-s)}\|x\| \text{ for } x\in \text{Ker}P,\,t\geqslant s.
\end{aligned}
\end{equation}

We define the function $h: \mathbb{R} \times X \to X$ by $h(t,u(t)):=|u(t)|^{k-1}u(t)+g(t,u(t))$. The equation (\ref{ex:2}) can now be rewritten as
\begin{equation}\label{ex:22}
\dfrac{du}{dt}=A(t)u(t)+h(t,u(t)) \text{ for } u(t)(\cdot)=w(t,\cdot), 
\end{equation}
where $\|h(\cdot,0)\|_{{C_b}(\mathbb{R}, X)}\leqslant\gamma + g(\cdot,0)$ for $\gamma:=\sup\limits_{t\in [0,\pi]}(\int_0^\pi|h(t,x)|^2dx)^{1/2}$. It is clearly to see that the function $h$ satisfying Assumption \ref{nonlinear} if $g$ satisfies also this assumption.
Therefore, we apply Theorem \ref{th:2} and Theorem \ref{stable} to get the well-posedness and exponential stability of AP-, AAP- and PAP- mild solutions for equation \eqref{ex:22} (hence for origin equation \eqref{ex:2}) as follows
\begin{theorem}
Let $X=L^2([0,\pi])$ and $a(\cdot)$ be an almost periodic function, the following assertions hold
\begin{itemize}
\item[$(i)$]
If $g \in C_b(\mathbb{R}\times [0,\pi],X)$ be a given AP- (resp. PAP-) function satifying Assumption \ref{nonlinear}, $\gamma$ and the Lipschitz constants are small enough, equation \eqref{ex:22} (hence \eqref{ex:2}) has one and only one  AP- (resp. PAP-) mild solution $\hat {u}$ on a small ball of $AP(\mathbb{R}\times [0,\pi],X)$ (resp. $PAP(\mathbb{R}\times [0,\pi],X)$). Moreover, these mild solutions are exponential stable in the sense of Theorem \ref{stable}.
\item[$(ii)$]
If $g \in C_b(\mathbb{R}_+\times [0,\pi],X)$ be a given AAP- function satisfying Assumption \ref{nonlinear}. If $\|u_0\|_X$, $\gamma$ and the Lipschitz constants of functions are small enough, equation (\ref{ex:22}) (hence \eqref{ex:2}) has one and only one  AAP- mild solution $\hat {u}$ on a small ball of $AAP(\mathbb{R}_+\times [0,\pi],X)$. Moreover, the mild solution is exponential stable in the sense of Theorem \ref{stable}.
\end{itemize}
\end{theorem}

\end{document}